\documentclass[11pt]{amsart}

\usepackage{amsmath,amsfonts,amsthm,mathrsfs,amssymb, enumerate,comment}
\usepackage[all]{xy}
\usepackage[stable]{footmisc}
\usepackage{mathtools}
\usepackage{multirow}

\usepackage[
            breaklinks=true,
            pdftitle={},
            pdfauthor={ }]{hyperref}
\usepackage{color}

\newcommand{\C}{\mathbb{C}}
\newcommand{\R}{\mathbb{R}}

\newcommand{\Z}{\mathbb{Z}}

\renewcommand{\P}{\mathbb{P}}

\numberwithin{equation}{section}
\newtheorem{theorem}[equation]{Theorem}
\newtheorem*{theorem*}{Theorem}
\newtheorem{corollary}[equation]{Corollary}

\newtheorem{proposition}[equation]{Proposition}
\newtheorem{definition}[equation]{Definition}

\theoremstyle{remark}
\newtheorem{remark}[equation]{Remark}

\def\subsection{
	\refstepcounter{equation}
	\setcounter{subsection}{\arabic{equation}}
	\noindent \arabic{section}.\arabic{subsection}
}

\newcommand{\F}{\mathcal{F}}

\newcommand{\M}{\mathcal{M}}
\newcommand{\OO}{\mathcal{O}}

\newcommand{\bk}{\bar{k}}
\newcommand{\CC}{\tilde{C}}

\begin{document}

\title{Rationality of moduli spaces of stable bundles on curves over $\R$.}

\begin{abstract}
Let $C$ be a smooth, projective, geometrically irreducible 
curve defined over $\R$ such that $C(\R) = \emptyset$.
Let $r>0$ and $d$ be integers which are coprime. 
Let $L$ be a line bundle on $C$ which corresponds to an $\R$ 
point of ${\rm Pic}^d_{C/\R}$.
Let $\M_{r,L}$ be the moduli space of stable bundles on 
the complexification of $C$ of rank $r$ 
and determinant $L$. We classify birational 
types of $\M_{r,L}$ over $\R$. 
\end{abstract}

\author{Souradeep Majumder \ and  \ Ronnie Sebastian}

\address{Souradeep Majumder: Department of Mathematics, Indian Institute of Science Education and Research (IISER) Tirupati \\ Andhra Pradesh-517507, India}
\email{souradeep@iisertirupati.ac.in}

\address{Ronnie Sebastian:
Department of Mathematics\\ Indian Institute of Technology\\ Powai, Mumbai-400076, India.}
\email{ronnie@math.iitb.ac.in}

\maketitle

\section{Introduction\ \footnote{This version of the article is slightly different from the published version.}}

Let $C$ be a smooth projective curve defined 
over an algebraically closed field $\bk$. For 
a pair of integers $(r,d)$, with $r>0$, let $\M_{r,d}$ denote 
the moduli space parameterizing rank $r$, degree $d$ semistable vector bundles on $C$. 
It is interesting to study the rationality properties of these 
moduli spaces. Let $L$ denote a line bundle on $C$ of degree $d$. 
It was proved in \cite{ksh} that when $r$ and $d$ 
are coprime, the moduli space $\M_{r,L}$ is rational over $\bk$.
It is an open problem to decide whether or not 
$\M_{r,L}$ is rational when the rank and degree are not coprime.

In \cite{Ho} the author works with an infinite base field $k$,
not necessarily algebraically closed. 
Let $L$ be a line bundle corresponding to a $k$ point 
of ${\rm Pic}^d_{C/k}$. Then the moduli space $\M_{r,L}$ is 
a variety defined over $k$.
Under the additional hypothesis, that the curve $C$ has a $k$ rational point,
it is shown that the moduli space $\M_{r,L}$ is rational as a variety over 
$k$, see \cite[Theorem 6.1, Corollary 6.2]{Ho}.

In this article we consider the situation when the 
curve $C$ is defined over $\R$. Several authors have studied 
questions related to moduli spaces in this situation. 
We refer the reader to the introduction 
in \cite{biswas-huisman-hurtubise} and \cite{Sch}.
In \cite{Sch} the author studies the topology of $\M_{r,d}(\R)$.
We consider the following rationality problem. 
Fix integers $r>0$ and $d$ such that they are coprime.
Let $L$ be a line bundle of degree $d$
corresponding to an $\R$ rational point of the Picard scheme 
${\rm Pic}^d_{C/\R}$. Then the moduli 
space $\M_{r,L}$ is defined over $\R$. 
It is interesting to classify the birational types of these moduli 
spaces (for varying $L$) as varieties over $\R$. 
In view of \cite{Ho}, 
they are all rational if $C(\R)\neq \emptyset$.
In view of \cite{ksh} they are all rational after 
base change to $\C$.

We deal with the case when $C(\R)=\emptyset$. 
The main result we prove is the following. 
\begin{theorem}
Fix integers $r>0$ and $d$ such that they are coprime. 
Let $L$ be a line bundle on $C$ which
corresponds to an $\R$ point of ${\rm Pic}^d_{C/\R}$. 
\begin{enumerate}
	\item The following are equivalent.
	\begin{enumerate}
	 \item The moduli space $\M_{r,L}$ is rational as a variety 
		over $\R$
	\item $\M_{r,L}(\R)\neq \emptyset$
	\item $r$ is odd.
	\end{enumerate}
	\item Let $r$ be even. Then $\M_{r,L}(\R)=\emptyset$ and 
		the varieties $\M_{r,L}$, for varying $L$, are isomorphic to each other 
		as varieties over $\R$.
\end{enumerate}
\end{theorem}

{\bf Acknowledgements}: We thank the referee for pointing to us
\cite{Ho}. We were not aware of this work and  
an earlier version of this article contained results which were 
already known due to \cite{Ho}. The second author was 
partially supported by a DST-INSPIRE grant. This 
research was supported in part by the International 
Centre for Theoretical Sciences (ICTS) during a 
visit for participating in the program - 
Complex Algebraic Geometry (Code: ICTS/cag/2018/10)

\section{Main Results}

Let $C$ be a smooth projective algebraic curve over $\R$. 
Let $\CC:=C \times_{\R} \C$. Let $g$ denote the genus of $C$ 
and assume that $g\geq 2$. Complex conjugation induces an 
involution $\sigma:\CC \to \CC$. Coherent $\OO_{C}$-modules 
give rise to coherent $\OO_{\CC}$-modules with an involution. 
More precisely, for a coherent $\OO_{C}$-module $\F_0$, 
let $\F:=\F_0 \otimes_{\R} \C$ be the corresponding $\OO_{\CC}$-module. 
Then we have an isomorphism $\delta: \F \to \sigma^*\F$ 
satisfying $\sigma^*\delta \circ \delta = {\rm Id}$.

\begin{remark}\label{rem1}
Converse to the above, let $X$ be a variety defined over $\R$ and let 
$X_\C:=X\times_\R\C$. If there is a quasi-coherent
sheaf $\F$ on $X_\C$ and an isomorphism $\delta:\F\to \sigma^*\F$
satisfying $\sigma^*\delta \circ \delta = {\rm Id}$,
then it is easily checked that there is a quasi-coherent
sheaf $\F_0$ on $X$ such that $\F\cong \F_0\otimes_\R\C$.
\end{remark}

\begin{definition}\label{def-real-quat}
Bundles $\F$
over $\CC$ which are of the type $\F_0\otimes_\R \C$ will be 
called $\R$ bundles. 
Bundles $\F$ over $\CC$ with an isomorphism 
$\delta:\F\to \sigma^*\F$ such that $\sigma^*\delta\circ \delta=-{\rm Id}$
will be called quaternionic bundles. 
\end{definition}
The involution $\sigma : \CC \to \CC$ induces an involution 
$\tilde{\sigma} : \M_{r, d} \to \M_{r,d}$. 
This is given on $\C$ points by $[E]\mapsto [\sigma^*E]$.

We now state a few known results along with proofs
so as to make this article self-contained.

\begin{proposition}\label{realpoint}{(\cite[Proposition 3.1]{biswas-huisman-hurtubise})}
An $\R$ rational point in the moduli space of stable bundles corresponds 
to an $\R$ bundle or quaternionic bundle. 
\end{proposition}
\begin{proof}
Let $[E]\in \M_{r,d}$ be an $\R$ point. Then since $\tilde{\sigma}^*[E]=[E]$, 
there is an isomorphism $\delta:E\to \sigma^*E$. Let ${\rm Spec}\, A_0$ 
be an affine open subset of $C$ such that the restriction of $E$ 
to ${\rm Spec}\, A_0\otimes_\R\C$ is free. On this open subset the 
isomorphism $\delta$ can be represented by a $r\times r$ matrix 
with entries in $A:=A_0\otimes_\R\C$, denote this matrix by $T$. 
Since $E$ is stable, we have $\sigma^*\delta \circ \delta = \lambda\cdot{\rm Id}$
for some $\lambda\in \C^*$. On the affine open ${\rm Spec}\, A$ this 
implies that $\sigma(T)\cdot T=\lambda\cdot {\rm Id}$. Thus, 
we also have $T\cdot \sigma(T)=\lambda\cdot {\rm Id}$. Applying $\sigma$ 
to this equation we see that 
$$\sigma(T)\cdot T=\sigma(\lambda)\cdot {\rm Id}=\lambda \cdot {\rm Id}$$
that is, $\lambda\in \R$.
Scaling $\delta$ by $\sqrt{\vert\lambda\vert}$ we get that 
$\sigma^*\delta\circ \delta=\pm {\rm Id}$.
If $\sigma^*\delta\circ \delta={\rm Id}$ then by Remark 
\ref{rem1} $E$ is an $\R$ bundle. Otherwise it is a quaternionic
bundle.
\end{proof}

If $C(\R)=\emptyset$ then it may happen that an $\R$ point  
of the moduli space of stable bundles does not correspond to a $\R$ 
bundle on $C$. For example, take $C:={\rm Proj}(\R[x,y,z]/(x^2+y^2+z^2))$.
Then ${\rm Pic}^1_{C/\R}(\C)$ is just one 
point and so is forced to be an $\R$ point. 
This point corresponds to the line bundle $\OO(1)$
on $\tilde C$,
which is clearly not defined over $\R$.

\begin{proposition} \label{quater}{(\cite[Proposition 4.3]{biswas-huisman-hurtubise})}
Let $E$ be a quaternionic bundle on $\CC$ of rank $r$ and 
degree $d$. Then $d+r(1-g)$ is even. 
\end{proposition}
\begin{proof}
The isomorphism $\delta$ on $E$ induces an isomorphism $\delta^*$ on 
$H^0(\CC, E)$.
Then $\delta^*$ is complex antilinear and $\delta^{*}\circ \delta^* = -{\rm Id}$. 
From this it is easy to see that  $H^0(\CC, E)$ is even 
dimensional as a $\C$-vector space. Similarly  $H^1(\CC, E)$ 
is also even dimensional. The proposition now follows
from Riemann-Roch. 
\end{proof}

\subsection{{\bf $C(\R)=\emptyset$}}.\\\\
In what follows we consider the situation when $C(\R)=\emptyset$. 
As before $L$ corresponds to an $\R$ point of ${\rm Pic}^d_{C/\R}$. 
We also assume that ${\rm gcd}(r,d)=1$.
We emphasize that an $\R$ point could correspond to an $\R$ bundle 
or a quaternionic bundle. 

\begin{proposition} \label{even}{(\cite[Proposition 4.2]{biswas-huisman-hurtubise})}
Every $\R$ line bundle on $C$ is of even degree.
\end{proposition}
\begin{proof}
Let $L$ be an $\R$ bundle of rank one. Let $p$ be a 
$\C$ point of $C$. Consider the $\R$ bundle 
$L_0=\OO(p+\sigma(p))$. If necessary, after twisting 
by a sufficiently large power of $L_0$, we may assume 
that $L$ has a global section $s$, which is defined over $\R$. 
Since $\sigma^*s=s$, 
the divisor corresponding to this section is invariant under
the action of $\sigma$. This shows that degree of this divisor 
is even (because the components of the divisor come in 
pairs $\{p,\sigma(p)\}$ with $p\neq \sigma(p)$), and so the degree of $L$ is even.
\end{proof}

\begin{theorem}\label{main-R2}
Let $L$ be an $\R$ point of ${\rm Pic}^d_{C/\R}$ and 
assume that $\M_{r,L}$ has an $\R$ rational point. Then $\M_{r,L}$ 
is rational as a variety over $\R$. 
\end{theorem}
\begin{proof}
Let $X:=\M_{r,L}$ denote the moduli space and let $X_\C$ 
denote $X\times_\R\C$. Let $\phi:X_\C\dashrightarrow \P^n$
denote a rational map which is birational, which exists by \cite{ksh}. 
Let $Z$ denote the degeneracy locus of $\phi$. Then 
${\rm codim}_{X_\C}(Z)\geq 2$. Let $U:=X_\C\setminus (Z\cup \sigma(Z))$.
Restricting line bundles gives an isomorphism 
${\rm Pic}(X_\C)\stackrel{\sim}{\longrightarrow} {\rm Pic}(U)$.

It is well known 
that ${\rm Pic}(X_\C)\cong \Z$.
The map $\sigma:X_\C\to X_\C$ induces 
an isomorphism ${\rm Pic}(X_\C)\stackrel{\sim}{\to}{\rm Pic}(X_\C)$
since $\sigma\circ\sigma={\rm Id}$. Since a line bundle with 
global sections gets mapped to a line bundle with global sections,
we see that the unique ample generator gets mapped to itself.

Letting $M$ denote $\phi^*\mathcal{O}(1)$, we have just seen
that $\sigma^*M\cong M$. Note that $M$ is a line bundle on all of $X_\C$. 
Thus, we may choose a 
$\delta:M\to \sigma^*M$ such that 
$\sigma^*\delta\circ \delta=\pm {\rm Id}$. If 
$\sigma^*\delta\circ \delta=-{\rm Id}$, then restricting 
to a point in $X_\C$ which is invariant under $\sigma$ (such a point
obviously corresponds to an $\R$ rational point of $X_\R$), we get a contradiction. Thus, 
$\sigma^*\delta\circ \delta={\rm Id}$, and so there is 
a line bundle $M_0$ on $X$ such that $M=M_0\otimes_\R\C$.
Thus, we get a rational map $\phi_\R:X\dashrightarrow \P {\rm H}^0(X,M_0)$
which is birational.
\end{proof}

\begin{corollary}\label{cor-reallinebundle}
Let $L$ be an $\R$ line bundle on $C$. Then $\M_{r,L}$ 
is rational as a variety over $\R$.
\end{corollary}
\begin{proof}
There is a dominant rational map of $\R$ varieties 
\[\P {\rm H}^1(C,L^\vee)^{\oplus (r-1)}\dashrightarrow \M_{r,L}\]
This shows that the set of $\R$ points in $\M_{r,L}$ 
is Zariski dense. The corollary now follows from the 
preceding theorem.
\end{proof}

\begin{proposition}\label{no-real-points}
Let $r$ be odd. 
Let $L$ be a line bundle corresponding to an $\R$ rational point of ${\rm Pic}^d_{C/\R}$.
The moduli space $\M_{r,L}$ is rational as a variety over $\R$.
\end{proposition}
\begin{proof}
The isomorphism 
\[\M_{r,L} \to \M_{r,L^{\otimes(r+1)}}\]
given by $E\mapsto E\otimes L$, is defined over $\R$. 
Since $r$ is odd, the line bundle $L^{\otimes (r+1)}$ 
is an $\R$ bundle. The proposition follows 
from Corollary \ref{cor-reallinebundle}.
\end{proof}
 
\begin{proposition}\label{no-real-points-prop}
Let $r$ be even and $d$ be odd. 
Let $L$ be an $\R$ point of ${\rm Pic}^d_{C/\R}$.
Then $\M_{r,L}(\R)=\emptyset$ and so
it is not rational as a variety over $\R$.
\end{proposition}
\begin{proof}
Assume that $\M_{r,L}$ has an $\R$ point corresponding to a bundle $E$. 
Then ${\rm Pic}^d_{C/\R}$ 
also has an $\R$ point, which corresponds to a line bundle $L$. 
By Proposition \ref{even}, since $d$ is odd, $L$ must be quaternionic. 
Next note that if $E$ is an $\R$ bundle, then 
$L$ will also be an $\R$ bundle. Since $L$ is a quaternionic bundle,
the only possible $\R$ points in $\M_{r, L}$ correspond
to quaternionic bundles. By Proposition \ref{quater}, applied to 
$E$, which is quaternionic, we get $d+r(1-g)$ 
is even. Since $d$ is odd we see $r(1-g)$ is odd. 
This contradicts the hypothesis.
\end{proof}

The case $r$ is even and $d$ is even cannot arise 
since $r$ and $d$ are coprime.

Let $r$ be even and $d$ be odd so that $\M_{r,L}$ is not 
rational. It may happen that $\M_{r,L_1}$ is birational 
over $\R$ with $\M_{r,L_2}$. Let 
$G:={\rm Pic}^0_{C/\R}(\R)$ and let $G^0$ denote 
the connected component of the identity. Then 
$G/G^0$ is an abelian group of cardinality at most 2,
see \cite[Proposition 3.3]{Gross-Harris}.

\begin{proposition}\label{prop-birational-classes}
Let $r$ be even, $d$ be odd such that they are coprime. 
Then the $\M_{r,L}$ are isomorphic to each other as varieties over $\R$,
where $L$ varies over the $\R$ points in ${\rm Pic}^d_{C/\R}$.
\end{proposition}
\begin{proof}
Let $L_1$ and $L_2$ be line bundles corresponding to 
$\R$ points in ${\rm Pic}^d_{C/\R}$.
Assume there is $M\in G$ such that 
$L_1^{-1}\otimes L_2\cong M^{\otimes r}$. 
Then the map $E\mapsto E\otimes M$ defines 
an isomorphism over $\R$ between $\M_{r,L_1}$ and 
$\M_{r,L_2}$. Thus, if $L^{-1}_1\otimes L_2$ is trivial 
in $G/rG$, then $M_{r,L_1}$ and $M_{r,L_2}$ are isomorphic. 
It suffices to show that $G/rG$ has cardinality 1.

From the diagram
\begin{equation}\label{e1}
\xymatrix{
	0\ar[r] & G^0\ar[d]^{[r]} \ar[r] & G\ar[d]^{[r]} \ar[r] & G/G^0\ar[r]\ar[d]^{[r]} &0\\
	0\ar[r] & G^0 \ar[r] & G \ar[r] & G/G^0\ar[r] &0
 }
\end{equation}
using the surjectivity of the left vertical arrow
we see that the cokernel of the middle vertical arrow
is isomorphic to the cokernel of the right vertical arrow.

Since $L_1$ corresponds to an $\R$ point of 
${\rm Pic}^d_{C/\R}$ and $d$ is odd, by Proposition \ref{even}
we see that $L$ is a quaternionic bundle. By Proposition 
\ref{quater}, $d+1-g$ is even, which forces 
that $g$ is even. 
By \cite[Proposition 3.3 (1)]{Gross-Harris} the cardinality of $G/G^0$ 
is 1. It follows that $G/rG$ is the trivial group. 
\end{proof}
The above results can be summarized into the following theorem.

\begin{theorem}
Fix integers $r>0$ and $d$ such that they are coprime. 
Let $L$ be a line bundle on $C$ which
corresponds to an $\R$ point of ${\rm Pic}^d_{C/\R}$. 
\begin{enumerate}
	\item The following are equivalent.
	\begin{enumerate}
	 \item The moduli space $\M_{r,L}$ is rational as a variety 
		over $\R$
	\item $\M_{r,L}(\R)\neq \emptyset$
	\item $r$ is odd.
	\end{enumerate}
	\item Let $r$ be even. Then $\M_{r,L}(\R)=\emptyset$ and 
		the varieties $\M_{r,L}$, for varying $L$, are isomorphic to each other 
		as varieties over $\R$.
\end{enumerate}
\end{theorem}
\begin{proof}
(1) (a) $\Longleftrightarrow$ (b) is Theorem \ref{main-R2}. (b) $\Longrightarrow$ (c) 
follows from Proposition \ref{no-real-points-prop}. (c) $\Longrightarrow$  (a) 
is Proposition \ref{no-real-points}.\\\\
(2) The first assertion is Proposition \ref{no-real-points-prop} and the second 
assertion is Proposition \ref{prop-birational-classes}.
\end{proof}

\end{document}